\documentclass{article}
\usepackage{amssymb}
\usepackage{graphicx}
\usepackage{amsthm}
\usepackage{amsmath,amstext,amsopn}
 
\usepackage{color}
\newcommand{\Rd}{{\mathbb{R}^d}}
\newcommand{\RR}{{\mathbb{R}}}

\newcommand{\R}{\mathbb{R}}
\newcommand{\C}{\mathbb{C}}

\newcommand{\E}{\mathbb{E}}
\newcommand{\A}{\mathcal{A}}
\newcommand{\HH}{{\mathcal{H}}}
\newcommand{\re}{{\rm{Re\ }}}

\newtheorem{thm}{Theorem}
\newtheorem{lem}[thm]{Lemma}
\newtheorem{prop}[thm]{Proposition}
\newtheorem{cor}[thm]{Corollary}
\newtheorem{example}[thm]{Example}
\newtheorem{defi}[thm]{Definition}
\newtheorem{rem}[thm]{Remark}
\definecolor{kb}{rgb}{0,0,0}

\newcommand{\kb}[1]{{\textcolor{kb}{#1}}}

\begin{document}
\title{On Hardy
spaces of local and nonlocal operators\footnotetext{KB was partially supported by MNiSW through grant N N201 397137.
BD was supported by MNiSW through grant N N201 397137 and 
by the DFG through SFB-701 `Spectral Structures and Topological Methods in Mathematics'.
TL was supported by Agence Nationale de la Recherche through grant ANR-09-BLAN-0084-01 and by MNiSW through grant N N201 373136.
Affiliations:
KB: Department of Mathematics, Stanford University,
450 Serra Mall, Stanford, CA 94305, USA, and
Institute of Mathematics of Polish Academy of Sciences, and
Institute of Mathematics and Computer Science, Wroc{\l}aw University of Technology, 
bogdan@pwr.wroc.pl;
BD: 
Faculty of Mathematics, University of Bielefeld, and
Institute of Mathematics and Computer Science, Wroc{\l}aw University of Technology
Wybrze\.ze Wys\-pia\'n\-skiego 27,
50-370 Wroc{\l}aw, Poland, bdyda@pwr.wroc.pl;
TL:
Department of Statistics and Probability, Michigan State University,
619 Red Cedar Road, C413 Wells Hall,
East Lansing MI 48824-1027, USA, and
Laboratoire de Math\'ematiques, Universit\'e d'Angers, tluks@msu.edu
\\
MSC 2010: Primary 60J75; Secondary 60J50, 42B30, 31B25;
Keywords: fractional Laplacian, Hardy spaces,
conditional stable L\'evy process
}
}

\author{Krzysztof Bogdan
and Bart{\l}omiej Dyda
and Tomasz Luks
\footnotetext{
}
}

\maketitle
\begin{abstract}
We characterize (conditional) Hardy spaces 
of the Laplacian and the fractional Laplacian by using Hardy-Stein type identities.
\end{abstract}
\section{Introduction}\label{sec:intro}
\kb{We fix an arbitrary open set} 
$D\subset \Rd$ 
and a \kb{point}  $x_0\in D$.
For $p>0$ and $0<\alpha<2$ we consider the Hardy space $\HH^p(D,\alpha)$
of the fractional Laplacian $\Delta^{\alpha/2}$. \kb{Here}
\begin{equation} \label{eq:laplacian}
  \Delta^{\alpha / 2} 
u(x) =
 \lim_{\eta \rightarrow 0^+} 
     \int_{|y-x| > \eta}\!\!\!\!\!\! \A\,
    \frac{u(y) - u(x)}{|y-x|^{d+\alpha}}\,dy\,,
\end{equation}
${\mathcal   A}=\Gamma((d+\alpha)/2)/(2^{-\alpha}\pi^{d/2}|\Gamma(-\alpha/2)|)$ and
$\HH^p(D,\alpha)$ is defined as follows. Let $X$ be the isotropic $\alpha$-stable  L\'evy process,  
i.e. the symmetric L\'evy process on $\Rd$  with the L\'evy measure $\nu(dy)={\mathcal A}|y|^{-d-\alpha}dy$ and zero Gaussian part (\cite{MR2569321}).
Let $\E_x$ be the expectation for $X$ starting at $x\in \Rd$.
We define $$\tau_D=\inf\{t 
\kb{>0}: X_t\not\in D\},$$ the
first exit time of $X$ from $D$. 
A Borel function $u:\R^d\to\R$ is called $\alpha$--harmonic
on $D$ if for every open $U$ relatively compact in $D$ (denoted $U\subset\subset D$) we have
\begin{equation}
  \label{eq:eq:dharm}
 u(x) = \E_x u(X_{\tau_U})\,,\quad x\in U\, .
\end{equation}
We assume that the expectation is
absolutely convergent, in particular--finite.
Equivalently, $u$ is $\alpha$--harmonic on $D$ if $u$ is twice continuously differentiable on $D$, $\int_\Rd |u(y)|(1+|y|)^{-d-\alpha}dy<\infty$, and 
\begin{equation}
  \label{eq:dh}
\Delta^{\alpha / 2} u(x) = 0\,,\quad x \in D\,.
\end{equation}
We refer to \kb{\cite{MR1671973, MR2365478, MR1687746, MR1654115}} for this characterization and detailed
discussion of $\alpha$-harmonic functions, including structure theorems for nonnegative $\alpha$-harmonic functions, and explicit formulas for the Green function, Poisson kernel and Martin kernel of $\Delta^{\alpha/2}$ for the ball.  
\kb{The equivalence of various notions of harmonicity for more general Markov processes is proved in \cite{MR2515419}.}
 We also refer to 
\cite[p. 120]{MR2320691}, \kb{which shows by means of an example} 
why the mean value property (\ref{eq:eq:dharm}) is preferred over analogues of (\ref{eq:dh}) for harmonic functions of 
Markov processes.
The reader may verify, using
(\ref{eq:eq:dharm}) and the strong Markov property of $X$, that $\{u(X_{\tau_U})\}_{U\subset
  \subset D}$ is a martingale ordered by inclusion of sets $U$. In particular,
$\E_{x} |u(X_{\tau_U})|^p$ is non-decreasing in $U$, if $p\geq 1$.

\begin{defi} Let $0<p<\infty$.
We write $u\in \HH^p=\HH^p(D,\alpha)$, if $u$ is $\alpha$-harmonic on $D$ and
\begin{equation}
  \label{eq:2}
\|u\|_p:=\sup_{U\subset\subset D} \left(\E_{x_0} |u(X_{\tau_U})|^p\right)^{1/p}<\infty\,. 
\end{equation}
\end{defi}
\noindent
	The finiteness condition does not depend on the choice of $x_0\kb{\in D}$, because the function $U\ni x\mapsto \E_{x} |u(X_{\tau_U})|^p$ satisfies Harnack inequality for arbitrary (Borel) function $u$, \kb{see \cite[p.~17]{bib:Rm} or \cite[Lemma~2.1]{MR704544}.}
If $p\leq q$, then $\HH^p\supset \HH^q$. 

We say that nonnegative functions $f(u)$ and $g(u)$ are {\it comparable}, and write $f(u)\asymp g(u)$, if 
numbers $0<c\leq C<\infty$ exist such that
$cf(u)\leq g(u)\leq C f(u)$ for every $u$.

Let $G_D(x,y)$ be the Green function of $\Delta^{\alpha/2}$ for the Dirichlet problem on $D$.
\kb{The function is defined as follows. We let
\begin{align*}
  p_t(x)&=(2\pi)^{-d}\int_ \Rd e^{-t|\xi|^\alpha}e^{ix\cdot\xi}\,d\xi\,,\quad t>0\,,\;x\in
  \Rd\,,
\end{align*}
so that $p_{t}(y-x)$ is the time-homogeneous transition density function of $X$.
Then we use Hunt's formula to define the Dirichlet heat kernel of $\Delta^{\alpha/2}$ for $D$:
\begin{align*}
p_D(t,x,y)&=p_t(y-x)-\E^x[\tau_D<t;\, p_{t-\tau_D}(y-X_{\tau_D})],\quad t>0
,\,x,y\in \Rd \,,
\end{align*}
cf. \cite[Section~2.2]{MR1329992}
or \cite[Section~3]{MR3050510}. Finally, we let
\begin{align*}
G_D(x,y)&=\int_0^\infty p_D(t,x,y)dt, \quad x,y\in \Rd\,.
\end{align*}
It may happen that 
$G_D\equiv \infty$ on $D$. This is the case, e.g., 
if $d=1\le \alpha$ and $D=(-\infty, \infty)$. 
Such sets $D$ are not excluded from our considerations.
We also remark that purely analytic definition of $G_D$ may be found in \cite{MR0350027}. 
}

The reader may notice that (\ref{eq:2}) is far from being explicit because it involves the distribution of $X_{\tau_U}$ for all $U\subset D$. The following result and the exact formula for $\|u\|_p$ given in (\ref{eq:HpD}) below simplify this perspective.
\begin{thm}\label{thm:t1}
 If $1<p<\infty$, then $\|u\|^p_p$ is {\rm comparable} on
  ${\mathcal H}^p$ with
    \begin{equation}
      \label{eq:3}
|u(x_0)|^p+
\int_D G_D(x_0,y)\int_\Rd 
\frac{[u(z)-u(y)]^2 \,\big(\,|u(z)|\vee |u(y)| \,\big)^{p-2}}{|z-y|^{d+\alpha}}\,dz\,dy.
    \end{equation}
In fact, $u\in {\mathcal H}^p$ if and only if $u$ is $\alpha$-harmonic
in $D$ and the integral is finite. 
\end{thm}
\kb{Incidentally, if $G_D\equiv \infty$ on $D$, $1<p<\infty$ and $u\in {\mathcal H}^p$, then $u$ must be constant on $D$.}
We also describe \kb{conditional} Hardy spaces $\HH^p_h=\HH^p_h(D,\alpha)$,
where $h$ is a fixed \kb{$\alpha$-}harmonic function 
positive on $D$ and vanishing on $D^c$. 
The class 
$\HH^p_h$ is of
considerable interest because it directly relates to {\it ratios} of $\alpha$-harmonic functions, \kb{weighted $L^p$ integrability of $\alpha$-harmonic functions and Doob's $h$-transform}. 
\kb{We note in passing that Doob's conditioning also plays an important role in the study of}
the relative Fatou theorem for $\alpha$-harmonic functions \cite{MR2114264, MR2214140, MR1980119}, and in the theory of conditional \kb{$\alpha$-}stable \kb{L\'evy} processes \cite{MR1671973,MR2256481}.

We 
give similar characterizations for Hardy spaces of the classical Laplacian $\Delta$, \kb{too}: formula (\ref{eq:chp}) below is the celebrated Hardy-Stein identity but Theorem~\ref{thm:chp2}, which may be considered a {\it conditional} Hardy-Stein identity, is 
new, and may be interesting for its own sake. 

The paper is composed as follows.
In Section~\ref{sec:hp} we observe the formula
  \begin{equation}
    \label{eq:chH2}
\sup_{U\subset\subset D} \E_{x_0} u^2(X_{\tau_U})= |u(x_0)|^2+
\int_D G_D(x_0,y)\int_\Rd \!\A\,\frac{[u(z)-u(y)]^2}{|z-y|^{d+\alpha}}\,dz\,dy\,,
  \end{equation}
for the norm of $\HH^2$, and we extend  it  in Lemma~\ref{lem:lF} and Theorem~\ref{thm:t1} to $\HH^p$ for $p>1$. 
	The conditional Hardy spaces $\HH^p_h$ are characterized in Lemma~\ref{lem:rhp}, Theorem~\ref{thm:t2} and formula (\ref{eq:HpDh})  in Section~\ref{sec:hhp}, see also Remark~\ref{rem:cHs}. 
In Section~\ref{sec:clas} we state the results for the Laplacian: 
formula (\ref{eq:chp}) and Theorem~\ref{thm:chp2}. 
In Section~\ref{sec:F} we describe the norm of the Hardy spaces in terms of the Krickeberg decomposition for $p\geq 1$, and we prove a classical Littlewood-Paley inequality.
 
Formula (\ref{eq:chH2}) and its modifications (\ref{eq:HpD}, \ref{eq:HpDh}, \ref{eq:cHS}) below are the main subject of the paper, and they may be considered nonlocal or conditional extensions of the classical Hardy-Stein equality, for which we refer the reader to (\ref{eq:chp}) in Section~\ref{sec:clas} and \kb{to} \cite{MR1921096, 2009-Pavlovic, MR2575382}.

Our work was motivated by the notion of the quadratic variation of martingales, \kb{operator} carr\'e du champ, and the characterization of the classical and martingale Hardy and Bergman spaces (\cite{MR750829, MR1885764, MR2273672, MR2575382, MR2279769, MR1921096, 2011-TL-sm, MR1320508}).
The resulting technique should apply to Hardy spaces of operators and Markov processes much more general than the fractional Laplacian and the isotropic stable L\'evy process. The style of the presentation and the inclusion of both jump and continuous processes in the present paper is intended to clarify the methodology and indicate such extensions. Our development is \kb{mostly} analytic. In fact, the definitions of the Hardy spaces can be easily formulated analytically by using the harmonic measures of the Laplacian and the fractional Laplacian (\cite{MR850715,MR0350027}). A clarifying comparison of the conditional and the non-conditional cases is made at the end of Section~\ref{sec:clas}.

\section{Characterization of $\HH^{p}$}\label{sec:hp}

Consider an open \kb{set} $U\subset \subset D$ and a real-valued function $\phi:\,\Rd\to \R$ which
is $C^2$ in a neighborhood of $\overline{U}$ and satisfies
$\int_\Rd |\phi(y)|(1+|y|)^{-d-\alpha}\kb{dy}<\infty$.
Then $\Delta^{\alpha/2}\phi$ is bounded on $\overline{U}$, and
for every $x\in \Rd$ we have
\begin{equation}\label{eq:rphi}
\phi(x)=\E_x \phi(X_{\tau_U})-\int_U G_U(x,y) \Delta^{\alpha/2}\phi(y)\,dy\,.
\end{equation}
\kb{
Indeed, if $\phi$ is compactly supported and smooth in $\Rd$, i.e. it is
a {\it test function}, then \eqref{eq:rphi} follows from Dynkin\rq{}s formula, see also a brief semi-analytic proof given in \cite[Lemma~8 with $b=0$]{MR2892584}. 
For arbitrary function $\phi$ satisfying the assumptions stated before \eqref{eq:rphi},
let test functions $\phi_n$ converge to $\phi$
in 
$L^1(\Rd,(1+|y|)^{-d-\alpha}dy)$ and 
in $C^2$ on a neighborhood of $\overline{U}$. Then $\Delta^{\alpha/2}\phi_n\to \Delta^{\alpha/2}\phi$ uniformly on $U$ because we can use Taylor expansion 
with remainder 
of the second-order for the integrand   in \eqref{eq:laplacian} in a neighborhood of $\overline{U}$, and we also have $\Delta^{\alpha/2}u(x)=   \int_{U^c}
    u(y)\A |y-x|^{-d-\alpha}dy$ if $x\in U$ and $u$ is supported in $U^c$. We also note that the distribution of 
$X_{\tau_U}$ for the process $X$ starting at $x$ has the density function $z\mapsto \int_U G_U(x,y)\A |z-y|^{-d-\alpha}dy$  in the complement of each neighborhood of $\overline{U}$.
The fact is  known
as the Ikeda-Watanabe formula and follows immediately from \eqref{eq:rphi} for test functions (we also refer to \cite{MR0142153} for the original contribution and to \cite[Lemma~6 with $b=0$]{MR2892584} for a brief semi-analytic proof). 
We note that $\int_U G_U(x,y)\A |z-y|^{-d-\alpha}dy\le c(1+|z|)^{-d-\alpha}$ in the complement of each neighborhood of $\overline{U}$, see also  \cite[Lemma~7]{MR1438304}.
By bounded convergence in a neighborhood of $\overline{U}$ and by $L^1$ convergence elsewhere
we extend 
\eqref{eq:rphi} from $\phi_n$ to $\phi$.
}
\kb{
The reader interested in proving \eqref{eq:rphi} by means of the maximum principle of $\Delta^{\alpha/2}$ may also do so by using
\cite[Lemma 3.8 and the proof of Theorem
3.9]{MR1671973}.}
\begin{lem}\label{lem:H2U}
  If $u$ is $\alpha$-harmonic on $D$ and $U\subset\subset D$, then 
 \begin{equation}
    \label{eq:H2U}
    \E_{x_0}u^2(X_{\tau_U})
=u(x_0)^2+\int_U G_U(x_0,y)\int_\Rd \!\A\,\frac{[u(z)-u(y)]^2}{|z-y|^{d+\alpha}}\,dz\,dy\,.
  \end{equation}
\end{lem}
\begin{proof}
  If $\int_\Rd u(y)^2(1+|y|)^{-d-\alpha}\kb{dy}=\infty$, then 
$\int_\Rd [u(z)-u(y)]^2/|z-y|^{d+\alpha}\,dz=\infty$ for every
$y$. Also $\E_{x_0}u^2(X_{\tau_U})=\infty$, because the distribution
of $X_{\tau_U}$ has density function bounded below by a multiple of
$(1+|y|)^{-d-\alpha}$ in the complement of the neighborhood 
of $\overline{U}$, \kb{ see the discussion of \eqref{eq:rphi}}. Therefore in what
follows we may assume that $\int_\Rd u(y)^2(1+|y|)^{-d-\alpha}\kb{dy}<\infty$.
Since $u^2$ is \kb{$C^2$ on $D$},
$\Delta^{\alpha/2}(u^2)$ is bounded on $\overline{U}$.
By (\ref{eq:rphi}) with $\phi=u^2$, for $x\in \Rd$ we have
\[ 
\E_x u^2(X_{\tau_U})=u^2(x)+\int_U G_U(x,y) \Delta^{\alpha/2}(u^2)(y)\,dy\,.
\] 
For $y\in \overline{U}$, $z\in \Rd$, we have
$u^2(z)-u^2(y)-2u(y)[u(z)-u(y)]=[u(z)-u(y))]^2$. Since
$\Delta^{\alpha/2}u(y)=0$, we have
\begin{align*}
\Delta^{\alpha/2}u^2(y)&=
\Delta^{\alpha/2}u^2(y)-2u(y)\Delta^{\alpha/2}u(y)\\
&=
\lim_{\eta\to 0^+}  \!\!\!\!\!\!\!\!
\int\limits_{\{z\in \Rd:\,|z-y|>\eta\}}
\!\!\!\!\!\!\!\!\!\!\!\!\!\!\A\,
\frac{u^2(z)-u^2(y)-2u(y)[u(z)-u(y)]}{|z-y|^{d+\alpha}}\,dz\\
&=
\int_\Rd\!\A\,
\frac{[u(z)-u(y)]^2}{|z-y|^{d+\alpha}}\,dz\,,
\end{align*}
and (\ref{eq:H2U}) follows.
\qed \end{proof}

We obtain the description of $\HH^2$ aforementioned in Introduction.
\begin{cor}\label{cor:c1}
 If $u$ is $\alpha$-harmonic in $D$, then \eqref{eq:chH2} holds.
\end{cor}
\begin{proof}
\kb{Recall that $G_U(x,y)\uparrow G_D(x,y)$ as $U\uparrow D$. By the monotone convergence theorem we obtain the result, also if $G_D\equiv \infty$ on $D$}.
\qed \end{proof}
We conclude that $\HH^2$ consists of precisely all those functions $\alpha$-harmonic on $D$ for which the quadratic form
on the right hand side of (\ref{eq:chH2}) is finite.

We now consider 
\kb{arbitrary real number} $p>1$. 
We note that $x\mapsto |x|^p$ is convex on $\R$, with the derivative $pa|a|^{p-2}$ at $x=a$.
For $a$, $b\in\C$ we let
\begin{equation}\label{eq:defF}
 F(a,b) = |b|^p-|a|^p - p\overline{a}|a|^{p-2}(b-a)\,.
\end{equation}
We have $F(a,b)=|b|^p$ if $a=0$, and $F(a,b)=(p-1)|a|^p$ if $b=0$.
\kb{If $a, b\in\R$, then $F(a,b)$ is the second-order Taylor remainder of $\RR\ni x\mapsto |x|^p$, and, by convexity, $F(a,b)\geq 0$.}
\begin{example}
{\rm For (even) $p=2,4,\ldots$ and $a, b\in \R$, we have
\[ F(a,b) = b^p-a^p - pa^{p-1}(b-a)
 = (b-a)^2 \sum_{k=0}^{p-2} (k+1)b^{p-2-k} a^k.
\]}
\end{example}
Let
$\varepsilon$, $b$ and $a$ be real numbers. \kb{For $p>1$ we} define
\begin{equation}\label{eq:Feps}
F_\varepsilon(a,b) = 
\re F(a+i\varepsilon, b+i\varepsilon)=
|b+i\varepsilon|^p-|a+i\varepsilon|^p
  - pa|a+i\varepsilon|^{p-2}(b-a)\,.
\end{equation}
\kb{$F_\varepsilon (a,b)$ is the second-order Taylor remainder of $\RR\ni x\mapsto (x^2+\varepsilon^2)^{p/2}$, and, by convexity, $F_\varepsilon(a,b)\geq 0$ (see below).}
Of course, $F_\varepsilon(a,b)\to F(a,b)$ as $\varepsilon\to 0$.
\begin{lem}\label{lem:F}
For every $p>1$, we have
\begin{equation}\label{elem-ineq}
F(a,b)
\;\asymp\; (b-a)^2(|b|\vee |a|)^{p-2},\qquad
a, b\in\R .
\end{equation} 
If $p\in (1,2)$, then 
\begin{equation}
  \label{eq:ub}
\kb{0\le}F_\varepsilon(a,b)\leq \kb{\frac{1}{p-1}}F(a,b)\,,\qquad \varepsilon, a, b\in\R\,.
\end{equation}
\end{lem}
\begin{proof}
We denote $K(a,b)=(b-a)^2(|b|\vee |a|)^{p-2}$. For every $k\in \R$, $F(ka,kb)=|k|^p F(a,b)$ and
$K(ka,kb)=|k|^pK(a,b)$. 
If $a=0$, then (\ref{elem-ineq}) becomes equality, hence we may assume that $a\neq 0$, in fact -- that $a=1$. 
Let $f(b)=F(1,b)=|b|^p-1-p(b-1)$. We  will prove that 
\begin{equation}\label{elem-ineq-b1}
 c_p (b-1)^2(|b|\vee 1)^{p-2} 
\leq f(b)\leq
 C_p (b-1)^2(|b|\vee 1)^{p-2}.
\end{equation}
Since $f(1)=f'(1)=0$ and $f''(y)=p(p-1)|y|^{p-2}$ \kb{for $y\neq 0$,} we obtain
\[ 
 f(b)=p(p-1)\int_1^b\int_1^x|y|^{p-2}\,dy\,dx=
p(p-1)\int_1^b|y|^{p-2}(b-y)dy\,.
\] 
The first integral is over a simplex of area $(b-1)^2/2$, and
it is a monotone function of the simplex (as ordered by inclusion). For $b$
close to $1$ the integral is comparable to $(b-1)^2$. For large $|b|$
the (second) integral is comparable to $|b|^p$. This proves (\ref{elem-ineq-b1}), hence (\ref{elem-ineq}).
\kb{We now consider $F_\varepsilon$ for $\varepsilon\neq 0$ and $p>1$.}
Let
$$
f_\varepsilon(b)=F_\varepsilon(1,b)
=(b^2+\varepsilon^2)^{p/2}-(1+\varepsilon^2)^{p/2}-p(1+\varepsilon^2)^{(p-2)/2}(b-1)\,.
$$
We have $f_\varepsilon(1)=f'_\varepsilon(1)=0$ and \kb{
\[
f''_\varepsilon(y)=(y^2+\varepsilon^2)^{(p-4)/2}p[y^2(p-1)+\varepsilon^2]\kb{\ge 0}\,,\qquad y\in \RR\, .
\]
Therefore, 
\[
f_\varepsilon(b)= \int_1^b\int_1^x f''_\varepsilon(y)\,dy\,dx\ge 0\,.
\]
In fact we have
$$f''_\varepsilon(y)
\le
p[1\vee (p-1)](y^2+\varepsilon^2)^{(p-2)/2}\,,\quad y\in \RR.$$ 
We now let $1<p\le 2$. For $y\in \RR$ we obtain  $f''_\varepsilon(y)\le
p|y|^{p-2}$, hence
\begin{align*}
f_\varepsilon(b)&\leq p \int_1^b\int_1^x|y|^{p-2}\,dy\,dx
\kb{=\frac1{p-1} f(b)\,.}
\end{align*}}
If $a\neq 0$, then by \eqref{elem-ineq-b1},
\begin{align*}
F_\varepsilon(a,b)&=|a|^p\left[|b/a+i\varepsilon/a|^p-|1+i\varepsilon/a|^p
  - p|1+i\varepsilon/a|^{p-2}(b/a-1)\right]\\
&=|a|^pF_{\varepsilon/a}\left(1,b/a\right)=|a|^pf_{\varepsilon/a}(b/a)\\
&\kb{\le \frac1{p-1} |a|^pf(b/a)=\frac1{p-1}F(a,b)\,.}
\end{align*}
If $a=0$, then
\[
F_\varepsilon(a,b)=(b^2+\varepsilon^2)^{p/2}-|\varepsilon|^p
\leq |b|^p \kb{=F(a,b)\,,
}
\]
\kb{too, since $(\rho+\eta)^{p/2}-\rho^{p/2}=\int_\rho^{\rho+\eta}\frac{p}{2}(y+\eta)^{p/2-1}dy\le \frac{p}{2}\eta^{p/2}\le \eta^{p/2}$ for $\rho,\eta\ge 0$}. The proof of (\ref{eq:ub}) is complete.
\qed \end{proof}
\kb{To prepare for limiting arguments we make the following observation, which follows from Fatou's lemma and dominated convergence theorem.}
\begin{rem}\label{rem:dct}
\kb{\rm If $0\le f_n\to f$ $\mu$-a.e., $\int f_n d\mu$ is bounded and $f_n\le cf$ for some constant $c$, then $\int f_nd\mu\to \int fd\mu$ as $n\to \infty$. }
\end{rem}

\begin{lem}\label{lem:lF}
If $u$ is $\alpha$-harmonic in $D$, $U\subset\subset D$, and $p>1$, then
  \begin{equation}
    \label{eq:HpU}
    \E_{x_0}|u(X_{\tau_U})|^p=|u(x_0)|^p+\int_U G_U(x_0,y)
 \int_\Rd \!\A\,\frac{F(u(y), u(z))}{|z-y|^{d+\alpha}}\,dz\,dy\,.
  \end{equation}
\end{lem}
\begin{proof}
We proceed as in Lemma~\ref{lem:H2U}. \kb{In particular, if 
\[
\int_\Rd|u(y)|^p(1+|y|)^{-d-\alpha}dy=\infty,
\]
then $\E_{x_0}|u(X_{\tau_U})|^p=\infty$ and also, by Lemma~\ref{lem:F}, 
\[
\int_{\Rd}\frac{F(u(y), u(z))}{|z-y|^{d+\alpha}}dz=\infty
\] 
for every $y\in \Rd$. Therefore in what follows, we assume that $\int_\Rd
|u(y)|^p(1+|y|)^{-d-\alpha}dy<\infty$, or $\E_{x_0}|u(X_{\tau_U})|^p<\infty$.}
We first consider the case of $p\geq 2$ and apply
(\ref{eq:rphi}) to $\phi=|u|^p\in C^2(D)$.
For $y\in D$ we have $\Delta^{\alpha/2}u(y)=0$, hence
\begin{align*}
&\Delta^{\alpha/2}|u|^p(y)
= \Delta^{\alpha/2}|u|^p(y)-p u(y)|u(y)|^{p-2}\Delta^{\alpha/2}u(y)\\
&=
\lim_{\eta\to 0^+}
\int\limits_{\{z\in \Rd:\,|z-y|>\eta\}}\!\!\!\!\!\!\A\,
\frac{|u(z)|^p-|u(y)|^p-pu(y)|u(y)|^{p-2}[u(z)-u(y)]}{|z-y|^{d+\alpha}}\,dz\\
&=  
\int_\Rd\!\A\,
\frac{F(u(y),u(z))}{|z-y|^{d+\alpha}}\,dz.
\end{align*}
This and (\ref{eq:rphi}) yield (\ref{eq:HpU}) for $p\geq 2$. We now consider $1<p<2$. We note that $|u+i\varepsilon|^p\in C^2(D)$. As in the first part of the proof, 
\begin{align*}
\Delta^{\alpha/2}|u+i\varepsilon|^p(y)
&= \Delta^{\alpha/2}|u+i\varepsilon|^p(y)-p u(y)|u(y)+i\varepsilon|^{p-2}\Delta^{\alpha/2}u(y)\\
&=  
\int_\Rd
\!\A\,\frac{F_\varepsilon(u(y),u(z))}{|z-y|^{d+\alpha}}\,dz,
\end{align*}
hence
\begin{equation}\label{eq:epsto0}
\E_{x_0}|u(X_{\tau_U})+i\varepsilon|^p=|u(x_0)+i\varepsilon|^p+\int_U G_U(x_0,y)
 \int_\Rd \!\A\,\frac{F_\varepsilon(u(y),u(z))}{|z-y|^{d+\alpha}}\,dz\,dy.
\end{equation}
\kb{
By Jensen's inequality,
\[
\E_{x_0}|u(X_{\tau_U})+i\varepsilon|^p \leq \E_{x_0}(|u(X_{\tau_U})|+|\varepsilon|)^p\le 2^{p-1}\E_{x_0}( |u(X_{\tau_U})|^p + |\varepsilon|^p),
\]
which remains bounded as $\varepsilon\to 0$.
By Remark~\ref{rem:dct} and
Lemma~\ref{lem:F} applied to 
the right-hand side of   \eqref{eq:epsto0}
and by the dominated convergence theorem applied to its left-hand side we
obtain \eqref{eq:HpU}.
}
\qed \end{proof}

\begin{proof}[Proof of Theorem~\ref{thm:t1}]
Lemma~\ref{lem:F}, Lemma~\ref{lem:lF} and monotone convergence imply the comparability of $\|u\|_p^p$ and (\ref{eq:3}) with the same constants as in (\ref{elem-ineq-b1}), under the mere assumption that $u$ be $\alpha$-harmonic on $D$. In fact,
  \begin{equation}
    \label{eq:HpD}
    \|u\|_p^p=|u(x_0)|^p+\int_D G_D(x_0,y)
 \int_\Rd \!\A\,\frac{F(u(y), u(z))}{|z-y|^{d+\alpha}}\,dz\,dy\,.
  \end{equation}
\qed \end{proof}
We note that 
in many cases \kb{sharp two-sided estimates 
of $G_D$ are} known. For instance, if $D$ is a bounded \kb{open} $C^{1,1}$ \kb{set}
in $\R^d$ and $d>\alpha$, then
\[
 G_D(x_0,y)\asymp\delta_D(y)^{\alpha/2}|y-x_0|^{\alpha-d}\,,
\]
where $\delta_D(y):={\rm dist}(y, D^c)$, \kb{see \cite{MR1654824, MR1490808, MR2200508}}. 

Recall the definition of $F_\varepsilon$, (\ref{eq:Feps}), and the fact that $F_0 = F$ of (\ref{eq:defF}). Before moving to conditional Hardy spaces we record the following observation.
\begin{lem}\label{lem:contractions}
For every $p>1$ and $a_1$, $a_2$, $b_1$, $b_2$, $\varepsilon\in\R$, we have
\begin{align}
F_\varepsilon(a_1\wedge a_2, b_1\wedge b_2) &\leq F_\varepsilon(a_1,b_1) \vee F_\varepsilon(a_2,b_2)\label{eq:min},\\
F_\varepsilon(a_1\vee a_2, b_1\vee b_2) &\leq F_\varepsilon(a_1,b_1) \vee F_\varepsilon(a_2,b_2)\label{eq:max}.
\end{align}
In particular, $F(a\wedge 1, b\wedge 1) \leq F(a,b)$ and $F(a\vee (-1), b\vee (-1)) \leq F(a,b)$, for all $a,b\in \R$. The latter also extends to $K(a,b)=(b-a)^2(|b|\vee |a|)^{p-2}$.
\end{lem}
\begin{proof}
Let $\varepsilon\neq 0$. 
We claim that the function $b\mapsto F_\varepsilon(a,b)$ decreases on $(-\infty, a]$
and increases on $[a, \infty)$. To see this, we consider
\begin{equation}\label{eq:pF}
\frac{\partial F_\varepsilon}{\partial b} (a, b) 
 =
 pb(b^2+\varepsilon^2)^{p/2-1} - pa(a^2+\varepsilon^2)^{p/2-1} \,.
\end{equation}
The function $h(x)=px(x^2+\varepsilon^2)^{p/2-1}$ has derivative $h'(x)=p(x^2+\varepsilon^2)^{p/2-2}(x^2(p-1)+\varepsilon^2)>0$.
If follows that the difference in (\ref{eq:pF}) 
is positive if $b>a$ and negative if $b<a$. This proves our claim.

Furthermore the function $a\mapsto F_\varepsilon(a,b)$ decreases on $(-\infty, b]$ and increases on $[b, \infty)$, as follows from calculating the derivative,
\[ 
\frac{\partial F_\varepsilon}{\partial a} (a,b)
 = p(a-b)(a^2+\varepsilon^2)^{p/2-2}(a^2(p-1)+\varepsilon^2) \,.
\] 
We now prove (\ref{eq:min}).
If $b_1\wedge b_2 = b_1$ and $a_1\wedge a_2 = a_1$ (or
$b_1\wedge b_2 = b_2$ and $a_1\wedge a_2 = a_2$), then (\ref{eq:min}) is trivial.
If $b_1\wedge b_2 = b_1$ and  $a_1\wedge a_2 =a_2$, then the monotonicity of $F_\varepsilon$ yields
\begin{align*}
F_\varepsilon(a_2, b_1) &\leq F_\varepsilon(a_1, b_1), \quad \textrm{if $b_1<a_2$},\\
F_\varepsilon(a_2, b_1) &\leq F_\varepsilon(a_2, b_2), \quad \textrm{if $b_1\geq a_2$}.
\end{align*}
The case $b_1\wedge b_2 = b_2$ and  $a_1\wedge a_2 =a_1$ obtains by renaming the arguments. This proves inequality (\ref{eq:min}).
(\ref{eq:max}) follows from (\ref{eq:min}) and the identity
\[
  F_\varepsilon(-a,-b) =F_\varepsilon(a,b). 
\]
The case $\varepsilon=0$ obtains by passing to the limit.
When $a=b$, we have $F(a,b)=0$, which yields the second  last statement of the lemma. For $K$ we obviously have 
 $(b\wedge 1-a \wedge 1)^2(|b\wedge 1|\vee |a\wedge 1|)^{p-2}\leq (b-a)^2(|b|\vee |a|)^{p-2}$ and
 $(b\vee (-1)-a \vee (-1) )^2(|b\vee (-1)|\vee |a\vee (-1)|)^{p-2}\leq (b-a)^2(|b|\vee |a|)^{p-2}$.
\qed \end{proof}

In passing we note that if the right-hand side of (\ref{eq:HpU}) is finite for  $u$, then it is also finite (in fact--smaller) for  $u\wedge 1$ and $u\vee (-1)$. The latter functions have smaller values and increments than $u$, a property defining {\it normal contractions} for Dirichlet forms (\cite{MR1303354}).

\section{Characterization of $\HH_h^{p}$}\label{sec:hhp}
The fractional Laplacian is a nonlocal operator and the corresponding stochastic process $X$ has jumps. In consequence the definitions of $\alpha$-harmonicity  (\ref{eq:eq:dharm}) and (\ref{eq:dh})
involve the values of the function on the whole of $D^c$ (\cite{MR2365478}). This is somewhat unusual compared with the classical theory of the Laplacian and the Brownian motion, and efforts were made by various authors to ascribe genuine boundary conditions to such processes and functions (\cite{MR2006232, MR1980119, MR2214908, MR2214140, MR2114264, 2013-TL-POTA}, see also \cite{MR2198018,MR2088043}).
One possibility is to study the boundary behavior of \kb{$\alpha$-}harmonic functions after an appropriate normalization.  
\kb{We shall use Doob's conditioning 
to normalize.
The procedure was proposed for classical harmonic functions in \cite{MR0109961}, and 
\cite[Chapter~11]{MR2152573} treats a general case.}
We shall focus on $\alpha$-harmonic functions vanishing on $D^c$, so that $D^c$ may be ignored.
Namely, let $h$ be \kb{$\alpha$-}harmonic 
and positive on $D$, and let $h$ vanish on $D^c$. 
\kb{Such functions are called singular $\alpha$-harmonic on $D$ (\cite{MR2365478}).}
We consider the transition semigroup
\begin{equation}\label{eq:cs}
  P^h_t f(x)=\frac{1}{h(x)}\int p_D(t,x,y)f(y)h(y)dy\,,
\end{equation}
where $p_D$, defined in Section~\ref{sec:intro},
is the time-homogeneous  transition density of $X$ {\it killed} on leaving
$D$ (\cite{MR1671973}). The semigroup property of $P_t^h$ follows directly from the Chapman-Kolmogorov  equations for $p_D$ \kb{(cf. \cite[Section~2.2]{MR1329992} or \cite[Section~3]{MR3050510})},
\[
\int_\Rd p_D(s,x,y)p_D(t,y,z)dy=p_D(s+t,x,z)\,.
\]
By $\alpha$-harmonicity and the optional stopping theorem, 
$\E_x h(X_{\tau_U\wedge t})=h(x)$, if $x\in U\subset\subset D$. Letting $U\uparrow D$, by Fatou's lemma we obtain
$\int p_D(t,x,y)h(y)dy=\E_x \{t<\tau_D:\, h(X_t)\}\leq h(x)$, i.e. $P^h_t$ is subprobabilistic.

The conditional process is defined as the Markov process with the transition semigroup $P^h$, and it will be denoted by the same symbol $X$. We let $\E_x^h$ be the corresponding expectation for $X$ starting at $x\in D$, 
\[
\E^h_xf(X_t)=\frac{1}{h(x)}\E_x [t<\tau_D;\,f(X_t)h(X_t)]\,,
\]
see also \cite{MR1671973}. 
A Borel function $r:D\to\R$ is $h$--harmonic
(on $D$) \kb{if} 
for every open $U\subset\subset D$ we have
\[ 
 r(x) = \E_x^h r(X_{\tau_U})=\frac{1}{h(x)}\E_x[X_{\tau_U}\in D; r(X_{\tau_U})h(X_{\tau_U})]\,,\quad x\in U.
\] 
Here we assume that the expectation is
absolutely convergent, in particular--finite.
It is evident that $r$ is $h$-harmonic if and only if $r=u/h$ on $D$, where $u$ is 
$\alpha$-harmonic on $D$ \kb{and vanishes on $D^c$}. 
\kb{Below we call such functions $u$ singular $\alpha$-harmonic on $D$, too, without requiring nonnegativity.}
We are interested in $L^p$ integrability of $u/h$, 
\kb{which} amounts to weighted $L^p$ integrability of $u$. The following definition is adapted from \cite{MR2606956}.
\begin{defi}\label{def:cHn}
For $0<p < \infty$ we define $\HH^p_h=\HH^p_h(D,\alpha)$
as the class of all the functions { $u:\,\Rd\to \R$, singular $\alpha$-harmonic} on $D$ and such that
\[
\| u \|^p_{\HH^p_h}:= 
\sup_{U\subset\subset D} \E^h_{x_0} \left| \frac{u(X_{\tau_U})}{h(X_{\tau_U})}
\right|^p 
=
\frac{1}{h(x_0)}\sup_{U\subset\subset D} \E_{x_0}\frac{|u(X_{\tau_U})|^p}{h(X_{\tau_U})^{p-1}}
< \infty\,,
\]  
where $\E_{x_0}|u(X_{\tau_U})|^p/h(X_{\tau_U})^{p-1}$ means $\E_{x_0}\big[X_{\tau_U}\in D; |u(X_{\tau_U})|^p/h(X_{\tau_U})^{p-1}\big]$.
\end{defi}
\noindent

By Harnack inequality, $\HH^p_h$
does not depend on the choice of $x_0\in D$.
\kb{In what follows we adopt the convention that $u(z)/h(z)=0$ if $u$ is singular $\alpha$-harmonic on $D$
and $z\in D^c$.}
\begin{rem}\label{rem:cHs}
{\rm
Note that the elements of this $\HH^p_h$ are $\alpha$-harmonic, rather than $h$-harmonic.
In view of Definition~\ref{def:cHn}, the {genuine conditional} Hardy space of $\Delta^{\alpha/2}$ \kb{and $h$} is 
$\{{u}/{h}:\,u\in \HH_h^p\}$, with the norm 
$\|u/h\|=\|u\|_{\HH^p_h}$. $\HH^p_h$ may be considered a {weighted Hardy space} of $\Delta^{\alpha/2}$, but it is convenient to call it {conditional Hardy space}, too. 
Below we focus on 
$\|u\|_{\HH^p_h}$, which yields description of both spaces.
}
\end{rem}

\begin{lem}  
  If $u$ is singular $\alpha$-harmonic on $D$ and $U\subset\subset D$, then
\begin{equation}\label{eq:H2Uh}
 \E_{x_0}\frac{u^2(X_{\tau_U})}{h(X_{\tau_U})}\!=\!
\frac{u(x_0)^2}{h(x_0)}
+\int\limits_U G_U(x_0,y)
\int\limits_\Rd\!\A\,\left[\frac{u(z)}{h(z)}-\frac{u(y)}{h(y)}\right]^2
\frac{h(z)\,dz\,dy}
{|z-y|^{d+\alpha}}\,.
 \end{equation}
\end{lem}
\begin{proof}
\kb{As in Lemma~\ref{lem:H2U} we assume that $\E_{x_0}u^2(X_{\tau_U})/h(X_{\tau_U})<\infty$,
equivalently $\int_D u^2(y)h(y)^{-1}(1+|y|)^{-d-\alpha}dy<\infty$, else both sides of \eqref{eq:H2Uh} are infinite. We also note that $u^2/h$ is  $C^2$ on $D$.}
Let $y\in D$. For arbitrary $z\in \Rd$ we have
\begin{align}
&\nonumber
\frac{u^2(z)}{h(z)}-  
\frac{u^2(y)}{h(y)}- 
2\frac{u(y)}{h(y)}(u(z)-u(y))+
\frac{u^2(y)}{h^2(y)}(h(z)-h(y))\\
&=\left[\frac{u(z)}{h(z)}-\frac{u(y)}{h(y)}\right]^2h(z)\,.\label{eq:rqf}
\end{align}
By (\ref{eq:rqf}) and $\alpha$-harmonicity,
\begin{align}
\nonumber
\Delta^{\alpha/2}\left(\frac{u^2}{h}\right)(y)&=
\Delta^{\alpha/2}\left(\frac{u^2}{h}\right)(y)
-2\frac{u(y)}{h(y)}\Delta^{\alpha/2}u(y)+
\frac{u^2(y)}{h^2(y)}\Delta^{\alpha/2}h(y)\\
&=
\int_\Rd
\!\A\,\left[\frac{u(z)}{h(z)}-\frac{u(y)}{h(y)}\right]^2|z-y|^{-d-\alpha}h(z)\,dz\,.\label{eq:lu2h}
\end{align}
Noteworthy, the integrand is nonnegative.
Following (\ref{eq:rphi}),
for $u^2/h$ we get
\begin{align*}
\E_{x_0}\frac{u^2(X_{\tau_U})}{h(X_{\tau_U})}&=
\frac{u(x_0)^2}{h(x_0)}
+\int_U G_U(x_0,y)\Delta^{\alpha/2}\left(\frac{u^2}{h}\right)(y)\,\kb{dy}.
\end{align*}
By using (\ref{eq:lu2h}) we obtain (\ref{eq:H2Uh}).
\qed \end{proof}
We can interpret (\ref{eq:H2Uh}) in terms of 
$h$-conditioning and $r=u/h$ as follows,
\[ 
\E^h_{x_0}r(X_{\tau_U})^2=
r(x_0)^2
+\int_U \!\frac{G_U(x_0,y)}{h(x_0)h(y)}
\int_\Rd\!\A\,\frac{\left[r(z)-r(y)\right]^2}{|z-y|^{d+\alpha}}\,\frac{h(z)}{h(y)}
\,dz\, h^2(y)dy\,.
\] 
This is an analogue of (\ref{eq:chH2}), and also indicates the general situation.
For $p>1$ we consider the expressions of the form
\begin{equation*}
F\Big(\frac{a}{s}, \frac{b}{t}\Big)\,,\quad a,b\in\C, \,s,t>0\,, 
\end{equation*}
see (\ref{eq:defF}). 
By Lemma~\ref{lem:F} we have 
\begin{equation}\label{eq:asF}
 \kb{0\le}{F}(\frac{a}{s},\frac{b}{t}) \asymp \Big(\frac{b}{t}-\frac{a}{s}\Big)^2
     \Big(\frac{|b|}{t} \vee \frac{|a|}{s}\Big)^{p-2}\,,\qquad a,b\in\R ,\,s,t>0\,,
\end{equation}   
and the comparisons on the right of (\ref{eq:asF}) hold with the constants $c_p$ and $C_p$ of (\ref{elem-ineq-b1}). 
If $1<p<2$, then we \kb{also consider} 
\begin{equation*} 
F_\varepsilon\Big(\frac{a}{s}, \frac{b}{t}\Big)\,,\qquad \kb{\varepsilon, a, b\in\R,} \;s,t>0\,,
\end{equation*}
where $F_\varepsilon$ is defined in (\ref{eq:Feps}).
 By Lemma~\ref{lem:F} we have 
\begin{equation}\label{eq:Feab}
\kb{0\le}{F_\varepsilon}(\frac{a}{s},\frac{b}{t}) \leq 
\kb{\frac{1}{p-1}F(\frac{a}{s},\frac{b}{t})}\,,
\qquad a,b\in\R\,,\,s,t>0\, .
\end{equation}
\begin{lem}\label{lem:Fof4var}
For $p>1$, $a, b\in\C$ and $s, t>0$, we have
\begin{equation}\label{eq:F4}
 {F}\big(\frac{a}{s},\frac{b}{t}\big)= \frac{|b|^p}{t^{p}}-\frac{|a|^p}{ts^{p-1}}
     - \frac{p|a|^{p-2} \overline{a}(b-a)}{ts^{p-1}}
     + \frac{(p-1)|a|^p(t-s)}{ts^p}\,.
\end{equation}
\end{lem}
\begin{proof}
\kb{By the definition of $F$,
\[
{F}\big(\frac{a}{s},\frac{b}{t}\big)=\frac{|b|^p}{t^{p}}-\frac{|a|^p}{s^{p}}-\frac{p|a|^{p-2}\overline{a}b}{ts^{p-1}}+\frac{p|a|^p}{s^{p}}.
\]
We get the same quantity expanding the right-hand side of (\ref{eq:F4}):
\[
\frac{|b|^p}{t^{p}}-\frac{|a|^p}{ts^{p-1}}-\frac{p|a|^{p-2}\overline{a}b}{ts^{p-1}}+\frac{p|a|^p}{ts^{p-1}}+\frac{p|a|^p}{s^{p}}-\frac{|a|^p}{s^{p}}
-\frac{p|a|^p}{ts^{p-1}}+\frac{|a|^p}{ts^{p-1}}
\]
\[
=\frac{|b|^p}{t^{p}}-\frac{|a|^p}{s^{p}}-\frac{p|a|^{p-2}\overline{a}b}{ts^{p-1}}+\frac{p|a|^p}{s^{p}}.
\]
}
\qed \end{proof}
The homogeneity seen on the left-hand side of (\ref{eq:F4})  is an interesting feature for the right-hand side of (\ref{eq:F4}). We also like to note that \kb{for real arguments} $t{F}(a/s,b/t)$ is the 
\kb{second-}order Taylor \kb{remainder for} 
$(a,s)\mapsto |a|^p /s^{p-1}$ at $(b,t)$ and, of course, ${ F}_\varepsilon(a/s,b/t)\to F(a/s,b/t)$ as $\varepsilon\to 0$.
\kb{
\begin{cor}\label{cor:Fof4var}
For $p>1$, $a, b\in\R$, and $s, t, \varepsilon>0$, we have
\begin{align*}
 {F}_\varepsilon\big(\frac{a}{s},\frac{b}{t}\big)&= \frac{|b+i\varepsilon t|^p}{t^{p}}-\frac{|a+i\varepsilon s|^p}{ts^{p-1}}
     - \frac{p|a+i\varepsilon s|^{p-2} a(b-a)}{ts^{p-1}}\\
&  \kb{-}\frac{p|a+i\varepsilon s|^{p-2} \varepsilon^2 s(t-s)}{ts^{p-1}}
     + \frac{(p-1)|a+i\varepsilon s|^p(t-s)}{ts^p}\,.
\end{align*}
\end{cor}
\begin{proof} The result follows from \eqref{eq:F4} because by \eqref{eq:Feps} we have
\begin{align*}
{F}_\varepsilon\big(\frac{a}{s},\frac{b}{t}\big)&=\re {F}\big(\frac{a+i\varepsilon s}{s},\frac{b+i \varepsilon t}{t}\big)\,.
\end{align*}
\qed
\end{proof}
}

\begin{lem} \label{lem:rhp}
If $u$ is singular $\alpha$-harmonic on $D$, $U\subset\subset D$  and $p>1$, then
\[
 \E_{x_0}\frac{|u(X_{\tau_U})|^p}{h(X_{\tau_U})^{p-1}}=
\frac{|u(x_0)|^p}{h(x_0)^{p-1}}
+\int_U G_U(x_0,y)
\int_\Rd
{F}\left(\frac{u(y)}{h(y)}, \frac{u(z)}{h(z)}\right)
\,\A\,\frac{h(z)dz\,dy}
{|z-y|^{d+\alpha}}\,.
\]
\end{lem}
\begin{proof}
\kb{As in Lemma~\ref{lem:lF} we assume that $\E_{x_0}|u(X_{\tau_U})|^p/h(X_{\tau_U})^{p-1}<\infty$,
equivalently $\int_D |u(y)|^p h(y)^{1-p}(1+|y|)^{-d-\alpha}dy<\infty$, else both sides of the equality in the statement are infinite.}
If $p\geq 2$, then $|u|^p/h^{p-1}\in C^2(D)$. 
By (\ref{eq:rphi}),
\[ 
\E_{x_0}\frac{|u(X_{\tau_U})|^p}{h(X_{\tau_U})^{p-1}}=
\frac{|u(x_0)|^p}{h(x_0)^{p-1}}
+\int_U G_U(x_0,y)\Delta^{\alpha/2}\left(\frac{|u|^p}{h^{p-1}}\right)(y)\,dy.
\] 
By 
 $\alpha$-harmonicity of $h$ and $u$,
\begin{align*}
\nonumber
&\Delta^{\alpha/2}\left(\frac{|u|^p}{h^{p-1}}\right)(y)=
\Delta^{\alpha/2}\left(\frac{|u|^p}{h^{p-1}}\right)(y)
-\frac{p|u(y)|^{p-2} u(y)}{h(y)^{p-1}}\Delta^{\alpha/2}u(y)\\\nonumber
&\quad\quad\quad\quad\quad\quad\quad\quad
 +
\frac{(p-1)|u(y)|^p}{h(y)^p}\Delta^{\alpha/2}h(y)\\
\nonumber
&=
\lim\limits_{\eta\to 0^+} \!\!\!\!\!\!\!
\int\limits_{\{z\in \Rd:\, |z-y|>\eta\}}\!\!\!\!\!\!\!\!\!
\bigg[
\frac{|u(z)|^p}{h(z)^{p-1}}-  
\frac{|u(y)|^p}{h(y)^{p-1}}
 - \frac{p|u(y)|^{p-2} u(y)}{h(y)^{p-1}}(u(z)-u(y))\\\nonumber
&
+\frac{(p-1)|u(y)|^p}{h(y)^p}(h(z)-h(y))
\bigg]\A\,|z-y|^{-d-\alpha}\,dz.
\end{align*}
\kb{By Lemma~\ref{lem:Fof4var} with $a=u(y)$, $s=h(y)$, $b=u(z)$, $t=h(z)$, 
the above equals}
\begin{align*}
&
\int_\Rd
h(z)\, {F}(u(y)/h(y), u(z)/h(z))\,
\,\A\,|z-y|^{-d-\alpha}dz\,.
\end{align*}
This gives the result for $p\geq 2$.
If $1<p<2$ then we argue as follows. 
\kb{Let $\varepsilon>0$.}
By $\alpha$-harmonicity \kb{of $u$ and $h$}, $\Delta^{\alpha/2}\left({|u+i\varepsilon h|^ph^{1-p}}\right)(y)$ equals
\begin{align}
&
\lim\limits_{\eta\to 0^+} \!\!\!\!\!\!
\int\limits_{\{z\in \Rd:\, |z-y|>\eta\}}\!\!\!
\bigg[
\frac{|u(z)+i\varepsilon h(z)|^p}{h(z)^{p-1}}-  
\frac{|u(y)+i\varepsilon h(
\kb{y})|^p}{h(y)^{p-1}}\nonumber\\\nonumber
& - \frac{p|u(y)+i\varepsilon h(y)|^{p-2} u(y)}{h(y)^{p-1}}(u(z)-u(y))
\kb{-
\frac{p|u(y)+i\varepsilon h(y)|^{p-2} \varepsilon^2}{h(y)^{p-2}}(h(z)-h(y))}\\\nonumber
&
+\frac{(p-1)|u(y)+i\varepsilon h(y)|^p}{h(y)^p}(h(z)-h(y))
\bigg]\A\,|z-y|^{-d-\alpha}\,dz\,.
\end{align}
\kb{By Corollary~\ref{cor:Fof4var} with $a=u(y)$, $s=h(y)$, $b=u(z)$, $t=h(z)$, the above equals} 
\begin{align*}
&
\int_\Rd
h(z)\, {F}_\varepsilon(u(y)/h(y), u(z)/h(z))\,
\,\A\,|z-y|^{-d-\alpha}dz\,.
\end{align*}
\kb{By (\ref{eq:rphi}) we get
\begin{align*}
E_{x_0}\frac{|u(X_{\tau_U})+i\varepsilon h(X_{\tau_U})|^p}{h(X_{\tau_U})^{p-1}}= &
\frac{|u(x_0)+i\varepsilon h(x_0)|^p}{h(x_0)^{p-1}} \nonumber\\
& +\int_U G_U(x_0,y)\int_\Rd
{F}_\varepsilon\left(\frac{u(y)}{h(y)}, \frac{u(z)}{h(z)}\right)
\,\A\,\frac{h(z)dzdy}
{|z-y|^{d+\alpha}}\,. 
\end{align*}
We then proceed as in the proof of Lemma~\ref{lem:lF}, \kb{letting $\varepsilon \to 0$}, using (\ref{eq:Feab})}, 
\kb{Remark~\ref{rem:dct}, and the assumed finiteness of $\E_{x_0}|u(X_{\tau_U})|^p/h(X_{\tau_U})^{p-1}$ and $\E_{x_0}h(X_{\tau_U})$.}
\qed \end{proof}

\begin{thm}\label{thm:t2}
 Let $1<p<\infty$. For singular $\alpha$-harmonic \kb{functions} $u$ on $D$,  $\|u\|^p_{\HH^p_h}$ is comparable with
  \begin{equation*} 
\frac{|u(x_0)|^p}{h(x_0)^p} +
\int_D \frac{G_D(x_0,y)}{h(x_0)h(y)}\int_\Rd
     \Big(\frac{|u(z)|}{h(z)}\vee \frac{|u(y)|}{h(y)}\Big)^{p-2}
\left[\frac{u(z)}{h(z)}-\frac{u(y)}{h(y)}\right]^2
\frac{h(z)dz\, h^2(y)dy}{h(y)|z-y|^{d+\alpha}}\,.
  \end{equation*}
\end{thm}
\begin{proof}
The result follows from Lemma~\ref{lem:rhp} and (\ref{eq:asF}).
In fact,
  \begin{equation}
    \label{eq:HpDh}
    \|u\|_{\HH^p_h}^p=
\frac{|u(x_0)|^p}{h(x_0)^p}
+ \!\int_{D}\! \frac{G_{\kb{D}}(x_0,y)}{h(x_0)h(y)}
 \int_\Rd\!\!\! F\left(\frac{u(y)}{h(y)}, \frac{u(z)}{h(z)}\right)
\A\,\frac{h(z)dz\, h^2(y)dy}{h(y)|z-y|^{d+\alpha}}\,.
  \end{equation}
\qed \end{proof}
We remark in passing that for $h\equiv 1$ we obtain $\HH^p_h=\HH^p$. To rigorously state this observation, 
one should discuss conditioning by functions $h$ with nontrivial values on $D^c$.
\kb{In this connection we note that \cite{MR1654115} suggest that the stopped (rather than the killed) process  should be used to this end (see also \cite[Remark~11]{MR2365478} and \cite[Chapter~11]{MR2152573}).}
We will not embark on this endeavor, instead in the next section we fully discuss the conditional Hardy spaces of a local operator, in which case the values of $h$ on $D^c$ are irrelevant.

\section{Classical Hardy spaces}\label{sec:clas}

Here we describe the Hardy spaces and the conditional Hardy spaces of harmonic functions of the Laplacian $\Delta=\sum_{j=1}^d \partial^2/\partial x_j^2$. 
The former case has been widely studied in the literature, mainly for the ball and the half-space, but also for smooth and Lipschitz domains, see \cite{ MR1805196, 1998-Koosis, MR0290095, MR0473215, MR716504}. The characterization of the Hardy spaces in terms of  quadratic functions appeared in \cite{1961-Stein} and \cite{1967-Widman} for harmonic functions on the half-space in $\Rd$. The case of $D$ being the unit ball was studied in detail in \cite{2004-Stevic, 2009-Pavlovic}. For more general domains in $\Rd$ see \cite{MR2279769, MR0473215, MR716504}. 

Throughout this section we 
assume that $D\kb{\subset \Rd}$ is \kb{open and} {\it connected}, i.e. it is a domain, \kb{and $x_0\in D$}. For $0<p<\infty$, the classical Hardy space $H^p(D)$ may be defined as the family of all those functions $u$ on $D$ which are harmonic on $D$ (i.e. $u\in C^2(D)$ and $\Delta u(x)=0$ for $x\in D$) and satisfy
\[
\|u\|_{H^p}:=\sup_{U\subset\subset D}\bigg(\E_{x_0}|u(W_{\tau_U})|^p\bigg)^{1/p}<\infty.
\]
Here $W$ is the Brownian motion on $\Rd$ and $\tau_U=\inf\{t\geq 0:\,W_t\notin D\}$. 
For a positive harmonic function $h$ on $D$ and $0<p<\infty$ we consider the space $H^p_h(D)$ of all those functions $u$ harmonic on $D$ which  satisfy
\[
\|u\|^p_{H^p_h}:=\sup_{U\subset\subset D} \E^h_{x_0} \left| \frac{u(W_{\tau_U})}{h(W_{\tau_U})}
\right|^p =\frac{1}{h(x_0)}\sup_{U\subset\subset D}\E_{x_0}\frac{|u(W_{\tau_U})|^p}{h(W_{\tau_U})^{p-1}}<\infty,
\]
where $\E^h_x$ is the expectation for the conditional Brownian motion (compare Section~\ref{sec:hhp} or see \cite{1984-Doob}). Let $G_D$ be the classical Green function of $D$ for $\Delta$. If $1<p<\infty$ and $u$ is harmonic on $D$, then the following Hardy-Stein identity holds
\begin{equation}\label{eq:chp}
\|u\|^p_{H^p}=|u(x_0)|^p+p(p-1)\int_DG_D(x_0,y)|u(y)|^{p-2}|\nabla u(y)|^2dy.
\end{equation}
The identity (\ref{eq:chp}) obtains by taking $h\equiv 1$ in the next theorem. (\ref{eq:chp}) generalizes \cite[Lemma 1]{2004-Stevic} and \cite[Theorem 4.3]{2009-Pavlovic}, where the formula was given for the ball in $\Rd$, see also \cite{1099-PS-jlms}. We note that (\ref{eq:chp}) is implicit in \cite[Lemma 6]{MR2279769}, but apparently the identity did not receive enough attention for general domains.  

\kb{If sharp two-sided estimates 
of $G_D$ are known,
then} we obtain explicit estimate \kb{for $\|u\|_{H^p}$}. For instance, if $D$ is a bounded $C^{1,1}$ domain in $\R^d$ and $d\geq 3$, then $G_D(x_0,y)\asymp\delta_D(y)|y-x_0|^{2-d}$, where $\delta_D(y):={\rm dist}(y, D^c)$,  
\kb{ see \cite{MR0239264, MR842803} or \cite{MR1741527}}.
\kb{For Lipschitz domains we also refer to \cite{MR1741527}.}

\begin{thm}\label{thm:chp2}
If $1<p<\infty$ and $u$ is harmonic on $D$, then
\begin{equation}\label{eq:cHS}
\|u\|^p_{H^p_h}=\frac{|u(x_0)|^p}{h(x_0)^p}+p(p-1)\int_D\frac{G_D(x_0,y)}{h(x_0)h(y)}\;\left|\frac{u(y)}{h(y)}\right|^{p-2}\left|\nabla \frac{u}{h}(y)\right|^2\; h^2(y)dy.
\end{equation}
\end{thm}

The remainder of this section is devoted to the proof of Theorem~\ref{thm:chp2}. The reader interested mostly in (\ref{eq:chp}) is encouraged to carry out similar but simpler calculations for $h\equiv 1$ and $p>2$. We note that (\ref{eq:cHS}) is quite more general than (\ref{eq:chp}) because usually $u/h$ is not harmonic. The same remark concerns (\ref{eq:Duhpe}, \ref{eq:Duph2}) for general $h$ as opposed to (\ref{eq:Duhpe}, \ref{eq:Duph2}) for $h=1$, which is a classical result (\cite[VII.3]{MR0290095}).
We start with the following well-known Green-type equality.
Consider an open set $U\subset \subset D$ and a real-valued function $\phi:\,\Rd\to \R$ which is $C^2$ in a neighborhood of $\overline{U}$.
Then $\Delta\phi$ is bounded on $\overline{U}$, and
for every $x\in D$,
\begin{equation}\label{eq:crphi}
\phi(x)=\E_x \phi(W_{\tau_U})-\int_U G_U(x,y) \Delta\phi(y)\,dy\,,
\end{equation}
see, e.g., \cite[p. 133]{1965-Dynkin} for the proof.

\begin{lem}\label{lem:Duhp}
Let $\varepsilon\neq 0$ \kb{and $p>1$}, and let $u$ be harmonic on $D$. We have
\begin{equation}\label{eq:Duhpe}
\Delta \bigg[\bigg(\frac{u^2}{h^2}+\varepsilon^2\bigg)^{p/2}h\bigg]=p\bigg(\frac{u^2}{h^2}+\varepsilon^2\bigg)^{(p-4)/2}\bigg[(p-1)\frac{u^2}{h^2}+\varepsilon^2\bigg]\bigg|\nabla \frac{u}{h}\bigg|^2 h\,.
\end{equation}
If $u\neq 0$ or $p\geq 2$, then 
\begin{equation}\label{eq:Duph2}
\Delta\bigg(\frac{|u|^p}{h^{p-1}}\bigg)=p(p-1)\bigg|\frac{u}{h}\bigg|^{p-2}\bigg|\nabla\frac{u}{h}\bigg|^2h\,.
\end{equation}
\end{lem}
\begin{proof}
Denote $u_i=\partial u/\partial x_i$, $h_i=\partial h/\partial x_i$,  $u_{ii}=\partial^2u/\partial x_i^2$ and $h_{ii}=\partial^2h/\partial x_i^2$, $i=1,\ldots,d$.
The lemma results from straightforward calculations based on the following observations:
\begin{align*}
\nabla |u|^p&=\nabla\big(u^2)^{p/2}=p|u|^{p-2}u\nabla u\,,
\quad \mbox{ if $p\geq 2$ or $u\neq 0$}\,,\\
\frac{\partial^2}{\partial x_i^2}|u|^p&=p(p-1)|u|^{p-2}u_i^2+p|u|^{p-2}uu_{ii}\,,
\quad \mbox{ if $p\geq 2$ or $u\neq 0$}\,,\\
\nabla h^{1-p}&=(1-p)h^{-p}\nabla h\,,\\
\Delta(fg)&=f\Delta g+2\nabla f\circ \nabla g+g\Delta f\,.
\end{align*}
This yields (\ref{eq:Duph2}) if $p\geq 2$ or $u(x)\neq 0$ at the point $x$ where the derivatives are calculated (and so $|u|^ph^{1-p}$ is of class $C^2$ there). To prove (\ref{eq:Duhpe}) we let $\varepsilon\neq 0$, denote $f(x)=u^2/h^2+\varepsilon^2$, and use a few more identities:
\begin{align*}
\nabla \frac{u}{h}&=\frac{\nabla u}{h}-\frac{u\nabla h}{h^2}\,, \quad 
\nabla \bigg(\frac{u}{h}\bigg)^2=2\frac{u}{h}\nabla \frac{u}{h}\,,\\
\Delta \bigg(\frac{u}{h}\bigg)^2&=\frac{2|\nabla u|^2}{h^2}-\frac{8u\nabla u \circ \nabla h}{h^3}+\frac{6u^2|\nabla h|^2}{h^4}\,,\\
\nabla f^{p/2}&=\frac{p}{2}f^{p/2-1}\nabla \bigg(\frac{u}{h}\bigg)^2\,,\\
\Delta f^{p/2}&=
\frac{p(p-2)}{4}f^{p/2-2}\bigg|\nabla\bigg(\frac{u}{h}\bigg)^2\bigg|^2
+\frac{p}{2}f^{p/2-1}\Delta\bigg(\frac{u}{h}\bigg)^2
\,,\\
\Delta \bigg(f^{p/2}h\bigg)&=
\frac{p(p-2)}{4}f^{p/2-2}\bigg|\nabla\bigg(\frac{u}{h}\bigg)^2\bigg|^2h
+pf^{p/2-1}\bigg|\nabla\frac{u}{h}\bigg|^2h\,. 
\end{align*} 
\qed \end{proof}
Noteworthy, we obtained nonnegative expressions in (\ref{eq:Duhpe}) and (\ref{eq:Duph2}). Also, if $\varepsilon\to 0$, then $\Delta [(u^2/h^2+\varepsilon^2)^{p/2}h]\to 
\Delta(|u|^ph^{1-p})$ almost everywhere on $D$.

\begin{lem}\label{lem:chp}
If $u$ is harmonic on $D$, $U\subset\subset D$  and $p>1$, then
\[
\E_{x_0}\frac{|u(X_{\tau_U})|^p}{h(X_{\tau_U})^{p-1}}=\frac{|u(x_0)|^p}{h(x_0)^{p-1}}+p(p-1)\int_UG_U(x_0,y)\bigg|\frac{u(y)}{h(y)}\bigg|^{p-2}\bigg|\nabla \frac{u}{h}(y)\bigg|^2h(y)dy.
\]
\end{lem}
\begin{proof}
For $p\geq2$ we have $|u|^ph^{1-p}\in C^2(D)$ and the result follows from (\ref{eq:crphi}) and Lemma~\ref{lem:Duhp}. If $1<p<2$, then we consider  $u+i\varepsilon h$ in place of $u$ and we let $\varepsilon \to 0$. By (\ref{eq:Duhpe}), (\ref{eq:crphi})
and dominated convergence we obtain the result.
\qed \end{proof}
\begin{proof}[Proof of Theorem~\ref{thm:chp2}]
The conclusion follows from Lemma~\ref{lem:chp} and monotone convergence, after dividing by $h(x_0)$ and rearranging the integrand.
\qed \end{proof}
We observe very close similarities between the Hardy-Stein identities  and conditional Hardy-Stein identities discussed in this paper. Specifically, functions $u$ and $u/h$ undergo the same transformation under the integral sign. In each case we see the Green function (and jump kernels in the non-local case) appropriate for the given operator, and in the conditional case, $h^2(y)dy$ appears as a natural reference measure. 
We remark in passing that the framework of conditional semigroups (\ref{eq:cs}) should be convenient for such calculations in more general settings.

\section{Further results}\label{sec:F}

We now discuss the structure of $\HH^p$. 
We start with $p=1$. 
The following is a counterpart of the theorem of Krickeberg for martingales (\cite{MR745449}), and an extension of \cite[Theorem 1]{MR2606956},
where the result was proved for singular $\alpha$-harmonic functions 
on bounded Lipschitz \kb{open sets}.

\begin{lem}\label{lem:h1}
Let $u\in\HH^1$. There exist nonnegative functions $f$ and $g$ which are $\alpha$-harmonic on $D$, satisfy  $u=f-g$ and uniquely minimize $f(x_0)+g(x_0)$. In fact, $f(x_0)+g(x_0)=\|u\|_1$.  If $u$ is singular $\alpha$-harmonic on $D$, then so are $f$ and $g$.
If $1\leq p<\infty$ and $u\in\HH^p$, then $\|u\|^p_p=\|f\|^p_p+\|g\|^p_p$. 

\end{lem}
\begin{proof}
Let $U_n$ be open, $U_n\subset\subset U_{n+1}$ for $n=1,2\ldots$ and $\bigcup_nU_n=D$. Let $\tau_n=\tau_{U_n}$. We have
\[
\|u\|_1=\lim_{n\to\infty}\E_{x_0}|u(X_{\tau_n})|<\infty \,.
\]
Let $u^+=\max (u,0)$ and $u^-=\max(-u,0)$.
For $n=1,2,\ldots$, we define 
\[
f_n(x)=\E_x u^+(X_{\tau_n}) \,,\quad
g_n(x)=\E_x u^-(X_{\tau_n}) \,, \quad  x\in \Rd\,.
\]
Obviously, functions $f_n$ and $g_n$ are nonnegative on $\Rd$, and finite and $\alpha$-harmonic on $U_n$. We have $u=f_n-g_n$. Since $\tau_n\leq\tau_{n+1}$, for 
every $x\in \Rd$,
\begin{align*}
f_n(x)&=
\E_x\left[\E_{X_{\tau_n}}u(X_{\tau_{n+1}}) \,; \,u(X_{\tau_n})>0\right]
\leq
\E_x\left[\E_{X_{\tau_n}}u^+(X_{\tau_{n+1}})\right]
=f_{n+1}(x) \,,
\end{align*}
and $g_n(x)\leq g_{n+1}(x)$. We let $f(x)=\lim f_n(x)$ and  $g(x) =\lim g_n(x)$.
By the monotone convergence theorem, the mean value property (\ref{eq:eq:dharm}) holds for $f$ and $g$. We obtain
\[
 f(x_0)+g(x_0)=
\lim_{n\to\infty}\E_{x_0}|u(X_{\tau_n})|=\|u\|_1<\infty \,.
\]
In view of Harnack inequality we conclude that $f$ and $g$ are finite, hence $\alpha$-harmonic on $D$. Also, $u=f-g$. If $u$ vanishes on $D$, then so do $f$ and $g$. 
For the uniqueness, we observe that if $\tilde{f},\tilde{g}\geq 0$ are $\alpha$-harmonic on $D$, and $u=\tilde{f}-\tilde{g}$, then
$-\tilde{g}\le u\le \tilde{f}$, hence $f\leq \tilde{f}$ and $g\leq \tilde{g}$ by the construction of $f$ and $g$. Therefore $f(x_0)+g(x_0)\leq \tilde{f}(x_0)+\tilde{g}(x_0)$, and equality holds if and only if $f(x_0)=\tilde{f}(x_0)$ and $g(x_0)=\tilde{g}(x_0)$, henceforth $f=\tilde{f}$ and $g=\tilde{g}$. 

Let $p>1$ and suppose that $u\in\HH^p\subset \HH^1$. By Jensen\rq{}s inequality,
\[
f_n(x)^p\leq\E_x u^+(X_{\tau_n})^p,\quad g_n(x)^p\leq\E_x u^-(X_{\tau_n})^p,
\]
hence
\[
f_n(x)^p+g_n(x)^p\leq\E_x |u(X_{\tau_n})|^p.
\]
For $m<n$ we have
\[
\E_{x_0}(f_n(X_{\tau_m})^p+g_n(X_{\tau_m})^p)\leq\E_{x_0}\E_{X_{\tau_m}}|u(X_{\tau_n})|^p=\E_{x_0}|u(X_{\tau_n})|^p.
\]
Letting $n\to\infty$, we get
\[
\E_{x_0}(f(X_{\tau_m})^p+g(X_{\tau_m})^p)\leq\|u\|^p_p.
\]
Hence $\|f\|^p_p+\|g\|^p_p\leq\|u\|^p_p$. On the other hand, $f,g \geq 0$, hence
\[
\|u\|^p_p=\lim_{n\to\infty}\E_{x_0}|f(X_{\tau_n})-g(X_{\tau_n})|^p
\]
\[
\leq\lim_{n\to\infty}\E_{x_0}(f(X_{\tau_n})^p+g(X_{\tau_n})^p)=\|f\|^p_p+\|g\|^p_p.
\]
The proof is complete.
\qed \end{proof}
\noindent We note that $\|u\|^p_p=\|f\|^p_p+\|g\|^p_p$ has a trivial analogue for $L^p$ spaces.

\begin{lem}\label{lem:ch1}
Let $u\in\HH^1_h$. There are nonnegative functions $f, g\in \HH^1_h$ which satisfy $u=f-g$ and uniquely minimize $f(x_0)+g(x_0)$. In fact, $f(x_0)+g(x_0)=\|u\|_{\HH^1_h}h(x_0)$. If $1\leq p<\infty$ and $u\in\HH^p_h$, then $\|u\|^p_{\HH^p_h}=\|f\|^p_{\HH^p_h}+\|g\|^p_{\HH^p_h}$. 
\end{lem}

\begin{proof}
If $u\in\HH^1_h$, then $u$ is singular $\alpha$-harmonic on $D$, $u\in\HH^1$ and $\|u\|_{\HH^1_h}=h(x_0)^{-1}\|u\|_1$ (conditioning is trivial for $p=1$). 
By Lemma~\ref{lem:h1}, $u$ has the Krickeberg decomposition $u=f-g$, and $f,g$ are nonnegative and singular $\alpha$-harmonic on $D$. In particular $\|f\|_{\HH^1_h}=f(x_0)/h(x_0)$ and $\|g\|_{\HH^1_h}=g(x_0)/h(x_0)$ are finite.
The reader may easily verify the rest of the statement of the lemma, following the previous proof and using the conditional expectation $\E^h$.
\qed \end{proof}

\begin{rem}\label{rem:h1}
Analogues of  Lemma~\ref{lem:h1} and Lemma~\ref{lem:ch1} are true for the classical Hardy spaces $H^p(D)$ and $H^p_h(D)$ for connected $D$. 
\end{rem}

As an application of (\ref{eq:chp}) we give a short proof of the following Littlewood-Paley type inequality (see \cite{2006-Pavlovic}, were the result was given for the ball in $\RR^2$).
Recall the notation $\delta_D(y)={\rm dist}(y, D^c)$.
\begin{prop}\label{prop:LP}
Consider a domain $D\subset \Rd $, and let $p\geq2$. For every function $u$ harmonic on $D$ we have 
\[
\|u\|^p_{H^p}-|u(x_0)|^p\geq p(p-1)d^{2-p}2^{1-p}\int_DG_D(x_0,y)\delta_D(y)^{p-2}|\nabla u(y)|^pdy\,.
\]
\end{prop}
\begin{proof}
We may assume that $\|u\|_{H^p}<\infty$. In view of Lemma~\ref{lem:h1} and Remark~\ref{rem:h1}, $u=f-g$, where $f,g$ are positive and harmonic on $D$ and $\|u\|^p_{H^p}=\|f\|^p_{H^p}+\|g\|^p_{H^p}$. Clearly, $|u(x_0)|^p\leq f(x_0)^p+g(x_0)^p$, hence $\|u\|^p_{H^p}-|u(x_0)|^p\geq \|f\|^p_{H^p}-|f(x_0)|^p
+ \|g\|^p_{H^p}-|g(x_0)|^p$.
Furthermore, by Jensen's inequality,
\[
|\nabla u|^p\leq 2^{p-1}(|\nabla f|^p+|\nabla g|^p)\,.
\]
Recall the following gradient estimate for the nonnegative harmonic function~$f$, 
\[
 f(x)\geq |\nabla f(x)|\delta_D(x)/d\,, \quad x\in D\, ,
\]
(\cite[Exercise 2.13]{MR737190}, see also \cite{MR1683048}).
Here $d$ is the dimension. 
By (\ref{eq:chp}), 
\begin{align*}
\|f\|^p_{H^p}-|f(x_0)|^p&=p(p-1)\int_DG_D(x_0,y)|f(y)|^{p-2}|\nabla f(y)|^2dy\\
&\geq p(p-1)d^{2-p}\int_DG_D(x_0,y)\delta_D(y)^{p-2}|\nabla f(y)|^pdy\,,
\end{align*}
and a similar estimate holds for $g$.
\qed \end{proof}

{\bf Acknowledgements}:
We thank Rodrigo Ba\~nuelos and Jacek Zienkiewicz for comments on classical Hardy spaces and \kb{on} the present paper. We thank the referee for remarks and suggestions.
\def\cprime{$'$}


\begin{thebibliography}{10}

\bibitem{MR2088043}
Hiroaki Aikawa.
\newblock Fatou and {L}ittlewood theorems for {P}oisson integrals with respect
  to non-integrable kernels.
\newblock {\em Complex Var. Theory Appl.}, 49(7-9):511--528, 2004.

\bibitem{MR1805196}
Sheldon Axler, Paul Bourdon, and Wade Ramey.
\newblock {\em Harmonic function theory}, volume 137 of {\em Graduate Texts in
  Mathematics}.
\newblock Springer-Verlag, New York, second edition, 2001.

\bibitem{MR1683048}
Rodrigo Ba{\~n}uelos and Michael~M.H. Pang.
\newblock Lower bound gradient estimates for solutions of {S}chr\"odinger
  equations and heat kernels.
\newblock {\em Comm. Partial Differential Equations}, 24(3-4):499--543, 1999.

\bibitem{MR704544}
Richard~F. Bass and Michael Cranston.
\newblock Exit times for symmetric stable processes in {${\bf R}\sp{n}$}.
\newblock {\em Ann. Probab.}, 11(3):578--588, 1983.

\bibitem{MR2198018}
Richard~F. Bass and Dahae You.
\newblock A {F}atou theorem for {$\alpha$}-harmonic functions in {L}ipschitz
  domains.
\newblock {\em Probab. Theory Related Fields}, 133(3):391--408, 2005.

\bibitem{MR850715}
J\"urgen Bliedtner and Wolfhard Hansen.
\newblock {\em Potential theory}.
\newblock Universitext. Springer-Verlag, Berlin, 1986.
\newblock An analytic and probabilistic approach to balayage.

\bibitem{MR1438304}
Krzysztof Bogdan.
\newblock The boundary {H}arnack principle for the fractional {L}aplacian.
\newblock {\em Studia Math.}, 123(1):43--80, 1997.

\bibitem{MR1741527}
Krzysztof Bogdan.
\newblock Sharp estimates for the {G}reen function in {L}ipschitz domains.
\newblock {\em J. Math. Anal. Appl.}, 243(2):326--337, 2000.

\bibitem{MR2006232}
Krzysztof Bogdan, Krzysztof Burdzy, and Zhen-Qing Chen.
\newblock Censored stable processes.
\newblock {\em Probab. Theory Related Fields}, 127(1):89--152, 2003.

\bibitem{MR1671973}
Krzysztof Bogdan and Tomasz Byczkowski.
\newblock Potential theory for the {$\alpha$}-stable {S}chr\"odinger operator
  on bounded {L}ipschitz domains.
\newblock {\em Studia Math.}, 133(1):53--92, 1999.

\bibitem{MR2569321}
Krzysztof Bogdan, Tomasz Byczkowski, Tadeusz Kulczycki, Michal Ryznar, Renming
  Song, and Zoran Vondra{\v{c}}ek.
\newblock {\em Potential analysis of stable processes and its extensions},
  volume 1980 of {\em Lecture Notes in Mathematics}.
\newblock Springer-Verlag, Berlin, 2009.
\newblock Edited by Piotr Graczyk and Andrzej Stos.

\bibitem{MR1980119}
Krzysztof Bogdan and Bart{\l}omiej Dyda.
\newblock Relative {F}atou theorem for harmonic functions of rotation invariant
  stable processes in smooth domains.
\newblock {\em Studia Math.}, 157(1):83--96, 2003.

\bibitem{MR2892584}
Krzysztof Bogdan and Tomasz Jakubowski.
\newblock Estimates of the {G}reen function for the fractional {L}aplacian
  perturbed by gradient.
\newblock {\em Potential Anal.}, 36(3):455--481, 2012.

\bibitem{MR2365478}
Krzysztof Bogdan, Tadeusz Kulczycki, and Mateusz Kwa{\'s}nicki.
\newblock Estimates and structure of {$\alpha$}-harmonic functions.
\newblock {\em Probab. Theory Related Fields}, 140(3-4):345--381, 2008.

\bibitem{MR2320691}
Krzysztof Bogdan and Pawe{\l} Sztonyk.
\newblock Estimates of the potential kernel and {H}arnack's inequality for the
  anisotropic fractional {L}aplacian.
\newblock {\em Studia Math.}, 181(2):101--123, 2007.

\bibitem{MR2256481}
Krzysztof Bogdan and Tomasz {\.Z}ak.
\newblock On {K}elvin transformation.
\newblock {\em J. Theoret. Probab.}, 19(1):89--120, 2006.

\bibitem{MR1687746}
Zhen-Qing Chen.
\newblock Multidimensional symmetric stable processes.
\newblock {\em Korean J. Comput. Appl. Math.}, 6(2):227--266, 1999.

\bibitem{MR2515419}
Zhen-Qing Chen.
\newblock On notions of harmonicity.
\newblock {\em Proc. Amer. Math. Soc.}, 137(10):3497--3510, 2009.

\bibitem{MR3050510}
Zhen-Qing Chen, Panki Kim, and Renming Song.
\newblock Dirichlet heat kernel estimates for fractional {L}aplacian with
  gradient perturbation.
\newblock {\em Ann. Probab.}, 40(6):2483--2538, 2012.

\bibitem{MR1654824}
Zhen-Qing Chen and Renming Song.
\newblock Estimates on {G}reen functions and {P}oisson kernels for symmetric
  stable processes.
\newblock {\em Math. Ann.}, 312(3):465--501, 1998.

\bibitem{MR1654115}
Zhen-Qing Chen and Renming Song.
\newblock Martin boundary and integral representation for harmonic functions of
  symmetric stable processes.
\newblock {\em J. Funct. Anal.}, 159(1):267--294, 1998.

\bibitem{MR2200508}
Zhen-Qing Chen and Renming Song.
\newblock A note on the {G}reen function estimates for symmetric stable
  processes.
\newblock In {\em Recent developments in stochastic analysis and related
  topics}, pages 125--135. World Sci. Publ., Hackensack, NJ, 2004.

\bibitem{MR2152573}
Kai~Lai Chung and John~B. Walsh.
\newblock {\em Markov processes, {B}rownian motion, and time symmetry}, volume
  249 of {\em Grundlehren der Mathematischen Wissenschaften [Fundamental
  Principles of Mathematical Sciences]}.
\newblock Springer, New York, second edition, 2005.

\bibitem{MR1329992}
Kai~Lai Chung and Zhong~Xin Zhao.
\newblock {\em From {B}rownian motion to {S}chr\"odinger's equation}, volume
  312 of {\em Grundlehren der Mathematischen Wissenschaften [Fundamental
  Principles of Mathematical Sciences]}.
\newblock Springer-Verlag, Berlin, 1995.

\bibitem{MR745449}
Claude Dellacherie and Paul-Andr{\'e} Meyer.
\newblock {\em Probabilities and potential. {B}}, volume~72 of {\em
  North-Holland Mathematics Studies}.
\newblock North-Holland Publishing Co., Amsterdam, 1982.
\newblock Theory of martingales, Translated from the French by J. P. Wilson.

\bibitem{MR0109961}
J.~L. Doob.
\newblock Conditional {B}rownian motion and the boundary limits of harmonic
  functions.
\newblock {\em Bull. Soc. Math. France}, 85:431--458, 1957.

\bibitem{1984-Doob}
Joseph~L. Doob.
\newblock {\em Classical Potential Theory and Its Probabilistic Counterpart},
  volume 262 of {\em Die Grundlehren der Mathematischen Wissenschaften}.
\newblock Springer-Verlag, New York, 1984.

\bibitem{MR750829}
Richard Durrett.
\newblock {\em Brownian motion and martingales in analysis}.
\newblock Wadsworth Mathematics Series. Wadsworth International Group, Belmont,
  CA, 1984.

\bibitem{1965-Dynkin}
Eugene~B. Dynkin.
\newblock {\em Markov Processes, I}, volume 121 of {\em Die Grundlehren der
  Mathematischen Wissenschaften}.
\newblock Springer-Verlag, Berlin-G\"ottingen-Heidelberg, 1965.
\newblock Translated with the authorization and assistance of the author by J.
  Fabius, V. Greenberg, A. Maitra, G. Majone.

\bibitem{MR1303354}
Masatoshi Fukushima, Y{\=o}ichi {\=O}shima, and Masayoshi Takeda.
\newblock {\em Dirichlet forms and symmetric {M}arkov processes}, volume~19 of
  {\em de Gruyter Studies in Mathematics}.
\newblock Walter de Gruyter \& Co., Berlin, 1994.

\bibitem{MR737190}
David Gilbarg and Neil~S. Trudinger.
\newblock {\em Elliptic partial differential equations of second order}, volume
  224 of {\em Grundlehren der Mathematischen Wissenschaften [Fundamental
  Principles of Mathematical Sciences]}.
\newblock Springer-Verlag, Berlin, second edition, 1983.

\bibitem{MR2214908}
Qing-Yang Guan and Zhi-Ming Ma.
\newblock Reflected symmetric {$\alpha$}-stable processes and regional
  fractional {L}aplacian.
\newblock {\em Probab. Theory Related Fields}, 134(4):649--694, 2006.

\bibitem{MR0142153}
Nobuyuki Ikeda and Shinzo Watanabe.
\newblock On some relations between the harmonic measure and the {L}\'evy
  measure for a certain class of {M}arkov processes.
\newblock {\em J. Math. Kyoto Univ.}, 2:79--95, 1962.

\bibitem{MR716504}
David~S. Jerison and Carlos~E. Kenig.
\newblock Boundary value problems on {L}ipschitz domains.
\newblock In {\em Studies in partial differential equations}, volume~23 of {\em
  MAA Stud. Math.}, pages 1--68. Math. Assoc. America, Washington, DC, 1982.

\bibitem{MR2214140}
Panki Kim.
\newblock Relative {F}atou's theorem for {$(-\Delta)^{\alpha/2}$}-harmonic
  functions in bounded {$\kappa$}-fat open sets.
\newblock {\em J. Funct. Anal.}, 234(1):70--105, 2006.

\bibitem{1998-Koosis}
Paul Koosis.
\newblock {\em Introduction to $H_p$ Spaces}, volume 115 of {\em Cambridge
  Tracts in Mathematics}.
\newblock Cambridge University Press, Cambridge, 1998.
\newblock Second edition.

\bibitem{MR1490808}
Tadeusz Kulczycki.
\newblock Properties of {G}reen function of symmetric stable processes.
\newblock {\em Probab. Math. Statist.}, 17(2, Acta Univ. Wratislav. No.
  2029):339--364, 1997.

\bibitem{MR0350027}
Naum~S. Landkof.
\newblock {\em Foundations of modern potential theory}.
\newblock Springer-Verlag, New York, 1972.
\newblock Translated from the Russian by A. P. Doohovskoy, Die Grundlehren der
  mathematischen Wissenschaften, Band 180.

\bibitem{2011-TL-sm}
Tomasz Luks.
\newblock Hardy spaces for the {L}aplacian with lower order perturbations.
\newblock {\em Studia Math.}, 204:39--62, 2011.

\bibitem{2013-TL-POTA}
Tomasz Luks.
\newblock Boundary {B}ehavior of $\alpha$-{H}armonic {F}unctions on the
  {C}omplement of the {S}phere and {H}yperplane.
\newblock {\em Potential Anal.}, 39(1):29--67, 2013.

\bibitem{MR1885764}
Olivier Mazet.
\newblock A characterization of {M}arkov property for semigroups with invariant
  measure.
\newblock {\em Potential Anal.}, 16(3):279--287, 2002.

\bibitem{MR2114264}
Krzysztof Michalik and Micha{\l} Ryznar.
\newblock Relative {F}atou theorem for {$\alpha$}-harmonic functions in
  {L}ipschitz domains.
\newblock {\em Illinois J. Math.}, 48(3):977--998, 2004.

\bibitem{MR2606956}
Krzysztof Michalik and Michal Ryznar.
\newblock Hardy spaces for {$\alpha$}-harmonic functions in regular domains.
\newblock {\em Math. Z.}, 265(1):173--186, 2010.

\bibitem{2006-Pavlovic}
Miroslav Pavlovi\'c.
\newblock A short proof of an inequality of {L}ittlewood and {P}aley.
\newblock {\em Proc. Amer. Math. Soc.}, 134(12):3625--3627, 2006.

\bibitem{2009-Pavlovic}
Miroslav Pavlovi\'c.
\newblock Green's formula and the {H}ardy-{S}tein identities.
\newblock {\em Filomat}, 23(3):135--153, 2009.

\bibitem{MR2575382}
Miroslav Pavlovi{\'c}.
\newblock Hardy-{S}tein type characterization of harmonic {B}ergman spaces.
\newblock {\em Potential Anal.}, 32(1):1--15, 2010.

\bibitem{MR2273672}
Philip~E. Protter.
\newblock {\em Stochastic integration and differential equations}, volume~21 of
  {\em Stochastic Modelling and Applied Probability}.
\newblock Springer-Verlag, Berlin, 2005.
\newblock Second edition. Version 2.1, Corrected third printing.

\bibitem{bib:Rm}
Marcel Riesz.
\newblock Int{\'e}grales de {R}iemann-{L}iouville et potentiels.
\newblock {\em Acta Sci. Math. (Szeged)}, 9(1--1):1--42, 1938-1940.

\bibitem{1961-Stein}
Elias~M. Stein.
\newblock On the theory of harmonic functions of several variables. {II}.
  {B}ehavior near the boundary.
\newblock {\em Acta Math.}, 106:137--174, 1961.

\bibitem{MR0290095}
Elias~M. Stein.
\newblock {\em Singular integrals and differentiability properties of
  functions}.
\newblock Princeton Mathematical Series, No. 30. Princeton University Press,
  Princeton, N.J., 1970.

\bibitem{MR0473215}
Elias~M. Stein.
\newblock {\em Boundary behavior of holomorphic functions of several complex
  variables}.
\newblock Princeton University Press, Princeton, N.J., 1972.
\newblock Mathematical Notes, No. 11.

\bibitem{1099-PS-jlms}
Peter Stein.
\newblock On a theorem of {M}. {R}iesz.
\newblock {\em J. Lond. Math. Soc.}, 8:242--247, 1933.

\bibitem{MR1921096}
Stevo Stevi{\'c}.
\newblock On harmonic {H}ardy and {B}ergman spaces.
\newblock {\em J. Math. Soc. Japan}, 54(4):983--996, 2002.

\bibitem{2004-Stevic}
Stevo Stevi\'c.
\newblock On harmonic {H}ardy spaces and area integrals.
\newblock {\em J. Math. Soc. Japan}, 56(2):339--347, 2004.

\bibitem{MR2279769}
Manfred Stoll.
\newblock The {L}ittlewood-{P}aley inequalities for {H}ardy-{O}rlicz spaces of
  harmonic functions on domains in {$\Bbb R^n$}.
\newblock In {\em Potential theory in {M}atsue}, volume~44 of {\em Adv. Stud.
  Pure Math.}, pages 363--376. Math. Soc. Japan, Tokyo, 2006.

\bibitem{MR1320508}
Ferenc Weisz.
\newblock {\em Martingale {H}ardy spaces and their applications in {F}ourier
  analysis}, volume 1568 of {\em Lecture Notes in Mathematics}.
\newblock Springer-Verlag, Berlin, 1994.

\bibitem{MR0239264}
Kjell-Ove Widman.
\newblock Inequalities for the {G}reen function and boundary continuity of the
  gradients of solutions of elliptic differential equations.
\newblock {\em Math. Scand.}, 21:17--37, 1967.

\bibitem{1967-Widman}
Kjell-Ove Widman.
\newblock On the boundary behavior of solutions to a class of elliptic partial
  differential equations.
\newblock {\em Ark. Mat.}, 6(6):485--533, 1967.

\bibitem{MR842803}
Zhong~Xin Zhao.
\newblock Green function for {S}chr\"odinger operator and conditioned
  {F}eynman-{K}ac gauge.
\newblock {\em J. Math. Anal. Appl.}, 116(2):309--334, 1986.

\end{thebibliography}
\end{document}